\documentclass[12pt]{amsart}
\usepackage{fullpage,url,amssymb,enumerate,colonequals}
\usepackage[all,pdf]{xy} 

\usepackage{comment}
\usepackage{color}
\usepackage{charter,euler} 
\DeclareSymbolFont{bchoperators}{T1}{bch}{m}{n}
\SetSymbolFont{bchoperators}{bold}{T1}{bch}{b}{n}
\makeatletter
\renewcommand{\operator@font}{\mathgroup\symbchoperators}
\makeatother

\usepackage{titlesec} 
\titleformat{\section}{\normalfont\bfseries\filcenter}{\thesection}{1em}{}

\usepackage[pdftex]{epsfig}

\usepackage{todonotes}

\newcommand{\C}{\mathbb{C}}

\newcommand{\PP}{\mathbb{P}}
\newcommand{\Q}{\mathbb{Q}}

\newcommand{\PGL}{\operatorname{PGL}}

\newcommand{\Aut}{\operatorname{Aut}}

\newcommand{\Res}{\operatorname{Res}}

\newcommand{\tors}{\operatorname{tors}}

\newcommand{\id}{\operatorname{id}}

\newcommand{\To}{\longrightarrow}

\numberwithin{equation}{section}

\newtheorem{theorem}{Theorem}
\newtheorem{lemma}[theorem]{Lemma}
\newtheorem{corollary}[theorem]{Corollary}
\newtheorem{proposition}[theorem]{Proposition}

\theoremstyle{definition}

\newtheorem{example}[theorem]{Example}

\theoremstyle{remark}

\definecolor{darkgreen}{rgb}{0,0.5,0}
\definecolor{rem}{rgb}{0.8,0,0}
\definecolor{new}{rgb}{0.7,0,0.6}
\definecolor{reply}{rgb}{0,0,0.8}
\usepackage[
        colorlinks, citecolor=darkgreen,
        backref=page,
        pdfauthor={Hang Fu, Michael Stoll},
        pdftitle={Elliptic curves with common torsion $x$-coordinates
                  and hyperelliptic torsion packets},
]{hyperref}
\usepackage[alphabetic,backrefs,lite]{amsrefs} 

\allowdisplaybreaks

\begin{document}

\title[Common torsion $x$-coordinates and torsion packets]%
      {Elliptic curves with common torsion $x$-coordinates \\
       and hyperelliptic torsion packets}

\author{Hang Fu}
\address{Department of Mathematics,
         National Taiwan University,
         Taipei, Taiwan}
\email{drfuhang@gmail.com}
\urladdr{https://sites.google.com/view/hangfu}

\author{Michael Stoll}
\address{Mathematisches Institut,
         Universit\"at Bayreuth,
         95440 Bayreuth, Germany.}
\email{Michael.Stoll@uni-bayreuth.de}
\urladdr{http://www.mathe2.uni-bayreuth.de/stoll/}

\date{\today}

\begin{abstract}
  We establish a connection between torsion packets on curves of genus~$2$
  and pairs of elliptic curves realized as double covers of the projective
  line~$\PP^1_x$ that have many common torsion $x$-coordinates. This can be used to
  show that the set of common torsion $x$-coordinates has size at least~$22$
  infinitely often and has $34$ elements in some cases. We also explain how
  we obtained the current record example of a hyperelliptic torsion packet
  on a genus~$2$ curve.
\end{abstract}



\maketitle


\section{Introduction}

Let $E_1$ and~$E_2$ be elliptic curves over~$\C$, together with double covers
$\pi_1 \colon E_1 \to \PP^1$ and $\pi_2 \colon E_2 \to \PP^1$ such that
the origin of~$E_j$ is a ramification point of~$\pi_j$. It is known
(as a consequence of Raynaud's~\cite{Raynaud} proof of the Manin-Mumford
Conjecture; see for example~\cite{BogomolovTschinkel}*{Thm.~4.2})
that when $\pi_1(E_1[2]) \neq \pi_2(E_2[2])$, then the intersection
\[ I(\pi_1, \pi_2) = \pi_1(E_{1,\tors}) \cap \pi_2(E_{2,\tors}) \]
is finite, where $E_{j,\tors}$ denotes the set of torsion points on~$E_j$.
One can then ask how large this intersection can be, possibly depending on the size
of $\pi(E_1[2]) \cap \pi_2(E_2[2])$.
In~\cite{BogomolovFuTschinkel}*{Conjs.\ 2 and~12}
(see also~\cite{BogomolovFu}*{Conj.~1.2}),
it is conjectured that there should be a uniform bound for the size
of~$I(\pi_1, \pi_2)$ whenever it is finite. The recent
paper~\cite{DeMarcoKriegerYe} by DeMarco, Krieger, and Ye establishes the existence of such a
uniform bound in the case that $\#\bigl(\pi_1(E_1[2]) \cap \pi_2(E_2[2])\bigr) = 3$.
The very recent results by Dimitrov, Gao, Ge, Habegger, and
K\"uhne~\cites{DimitrovGaoHabegger,Kuehne,Gao,GaoGeKuehne} on uniformity
in the Mordell-Lang conjecture for subvarieties
of abelian varieties (see also below) now imply the existence of a uniform
bound for $\#I(\pi_1, \pi_2)$ as conjectured. This follows by applying
their result to the families of curves in $E_1 \times E_2$ obtained as the pull-back
of the diagonal in~$\PP^1 \times \PP^1$ under~$(\pi_1, \pi_2)$.
The bounds are (so far) not explicit, and so it is an interesting question
how large they have to be.

An at first sight somewhat different question is how large a torsion
packet on a curve of genus~$2$ over~$\C$ can be. Recall that a \emph{torsion packet}
on a curve~$C$ of genus at least~$2$ is a maximal set of points on~$C$ such that
the (linear equivalence class of the) difference of any two points in the set
is a point of finite order on the Jacobian of~$C$. Again by~\cite{Raynaud},
such a torsion packet is always finite, and one can ask for a bound on its
size that depends only on the genus~$g$ \cite{Mazur}*{top of page~234}.
In~\cite{Poonen} it is shown
that there are infinitely many essentially distinct curves of genus~$2$ with
a \emph{hyperelliptic} torsion packet (i.e., containing the Weierstrass points)
of size at least~$22$.
Within the family giving rise to these examples, there is (at least) one
with a torsion packet of size~$34$; see~\cite{Stoll34} and Section~\ref{S:tp} below.
Using the fact that in the case $\#\bigl(\pi_1(E_1[2]) \cap \pi_2(E_2[2])\bigr) = 3$,
the (desingularization of the) pull-back of the diagonal in~$\PP^1 \times \PP^1$
under $(\pi_1, \pi_2)$ is a bielliptic curve of genus~$2$, DeMarco, Krieger,
and Ye deduce the existence of a uniform bound on the size of hyperelliptic
torsion packets on bielliptic curves of genus~$2$, and the work of Dimitrov,
Gao, Ge, Habegger, and K\"uhne already mentioned above now establishes the existence
of a bound that depends only on~$g$, as expected (and much more).

\medskip

In this article, we use the connection between bielliptic genus~$2$ curves
and pairs $(\pi_1, \pi_2)$ such that $\#\bigl(\pi_1(E_1[2]) \cap \pi_2(E_2[2])\bigr) = 3$
that was already mentioned above,
together with a new observation that relates pairs $(\pi_1, \pi_2)$ with differing
sizes of $\pi_1(E_1[2]) \cap \pi_2(E_2[2])$ to establish a close relation
between the largest size of an intersection $I(\pi_1, \pi_2)$ and the
largest size of a hyperelliptic torsion packet on a bielliptic genus~$2$ curve.
Specifically, we show the following (see Corollary~\ref{cor:I-and-T}).

\begin{theorem} \label{Trel}
  The maximal size of a finite intersection~$I(\pi_1, \pi_2)$ is at least
  as large as the maximal size of a hyperelliptic torsion packet on
  a bielliptic genus~$2$ curve.
\end{theorem}

This follows from an explicit correspondence between bielliptic genus~$2$
curves and pairs $(\pi_1, \pi_2)$. Under this correspondence, the family of genus~$2$
curves studied in~\cite{Poonen} is related to a family of pairs of elliptic
curves with certain properties. As a consequence, we obtain the following result,
which proves Conjecture~23 in~\cite{BogomolovFuTschinkel}.

\begin{theorem} \label{T22}
  There are infinitely many essentially distinct pairs~$(\pi_1, \pi_2)$
  as above such that $\#I(\pi_1, \pi_2) \geq 22$.
\end{theorem}

Applying the correspondence to our example of a torsion packet of size~$34$,
we obtain the following new record for the size of~$I(\pi_1, \pi_2)$.

\begin{theorem} \label{T34}
  Let $s \in \C$ satisfy $s^8 + 174 s^4 + 81 = 0$. Consider
  \[ E_1 \colon y^2 = (x^2 - s^2)(x^2 - (1/s)^2) \qquad\text{and}\qquad
     E_2 \colon y^2 = (x^2 - (s/3)^2)(x^2 - (3/s)^2) \,.
  \]
  We take $\pi_1$ and~$\pi_2$ to be the $x$-coordinate map (and fix the origins
  of $E_1$ and~$E_2$ to be $(s,0)$ and~$(s/3,0)$, respectively). Then
  \[ I(\pi_1, \pi_2) = \pi_1(E_1[48]) \cap \pi_2(E_2[48]) \]
  and $\#I(\pi_1, \pi_2) = 34$.
\end{theorem}

Note that except for the fact that all the torsion points with common
$x$-coordinate have order dividing~$48$, the statement can be easily checked
by a computation, which shows in particular that $\#I(\pi_1, \pi_2) \ge 34$.

\medskip

The structure of this paper is as follows.
In Section~\ref{S:pairs}, we define a pair~$P(\pi_1, \pi_2)$ of numerical invariants
of~$(\pi_1, \pi_2)$ and set up a correspondence between pairs~$(\pi_1, \pi_2)$
and~$(\pi'_1, \pi'_2)$ whose invariants are related in a certain way.
This implies a relation between $\#I(\pi_1, \pi_2)$ and~$\#I(\pi'_1, \pi'_2)$.
In Section~\ref{S:relg2}, we explain the connection between pairs~$(\pi_1, \pi_2)$
such that $\#\bigl(\pi_1(E_1[2]) \cap \pi_2(E_2[2])\bigr) = 3$ and bielliptic
curves of genus~$2$. This connection, together with the correspondence
from Section~\ref{S:pairs} then implies Theorems~\ref{Trel} and \ref{T22}. In Section~\ref{S:tp}, we explain how the example~\cite{Stoll34}
of a large hyperelliptic torsion packet was obtained. Finally,
in Section~\ref{S:subfamily}, we give more details on the
pairs~$(\pi_1, \pi_2)$ with $\pi_1(E_1[2]) \cap \pi_2(E_2[2]) = \emptyset$
that correspond to the curves in the family considered by Poonen in~\cite{Poonen}. We give the proof of Theorem \ref{T34}, and also show that the curves $E_1$ and~$E_2$ that appear in Theorem~\ref{T34}
are isogenous, which is unnecessary for the proof, but gives a hint of why we are able
to get many common torsion $x$-coordinates in this way.

All geometric objects in this paper will be over the complex numbers
unless explicitly stated otherwise.

\subsection*{Acknowledgments} \strut

The first named author would like to thank Laura DeMarco for her valuable comments
to the first version of this article. The authors would like to thank Yuri Bilu
for connecting us with each other.
The authors would also like to thank the anonymous referees for their helpful feedback.


\section{Relations among various pairs $(\pi_1, \pi_2)$}\label{S:pairs}

Let $\pi \colon E \to \PP^1$ be a double cover such that $E$ is an elliptic
curve and the origin of~$E$ is a ramification point of~$\pi$.
We note that the set $\pi(E_{\tors})$ does not depend on which of the
four ramification points we choose as the origin, since the difference
of any two is a point of order~$2$. Given $\pi$ as above, the
action of~$E[2]$ on~$E$ by translation induces an action on~$\PP^1\!$.
We denote the isomorphic copy of the Klein Four Group inside~$\PGL(2)$
that is the image of~$E[2]$ by~$G(\pi)$.

Given two double covers $\pi_1$, $\pi_2$ as above, we can then classify
them according to the sizes of $\pi_1(E_1[2]) \cap \pi_2(E_2[2])$
and of $G(\pi_1) \cap G(\pi_2)$. Note that the first set is a union of
orbits under the second group, on which the group acts without fixed points.
If $\pi_1(E_1[2]) = \pi_2(E_2[2])$, then $E_1$ and~$E_2$ are isomorphic
(up to the choice of origin on the elliptic curves) and
$I(\pi_1, \pi_2) = \pi_1(E_{1,\tors}) = \pi_2(E_{2,\tors})$ is infinite.
Excluding this case, the possibilities for the pair
\[ P(\pi_1, \pi_2) = \bigl(\#(\pi_1(E_1[2]) \cap \pi_2(E_2[2])), \#(G(\pi_1) \cap G(\pi_2))\bigr) \]
are
\[ (3, 1), \quad (2, 1), \quad (1, 1), \quad (0, 1); \quad
   (2, 2), \quad (0, 2); \quad (0, 4) \,.
\]

We will now relate pairs with different invariants.

\begin{proposition} \label{prop:main}
  Let $\pi_1 \colon E_1 \to \PP^1$ and $\pi_2 \colon E_2 \to \PP^1$ be two
  double covers as above and fix a non-trivial element $\alpha \in G(\pi_1) \cap G(\pi_2)$
  (in particular, we assume that $G(\pi_1) \cap G(\pi_2)$ is non-trivial).
  For $j \in \{1,2\}$, we denote by $T_j \in E_j[2]$ the point such that
  $\pi_j(P + T_j) = \alpha(\pi_j(P))$, and we write $E'_j = E_j/\langle T_j \rangle$.
  Then there are double covers $\pi'_j \colon E'_j \to \PP^1$ and a morphism
  $\beta \colon \PP^1 \to \PP^1$ of degree~$2$ such that
    \[ \#\bigl(\pi_1(E_1[2]) \cap \pi_2(E_2[2])\bigr) = 2\#\bigl(\pi'_1(E'_1[2]) \cap \pi'_2(E'_2[2])\bigr)-4 \]
  and
  \[ I(\pi_1, \pi_2) = \beta^{-1}\bigl(I(\pi'_1, \pi'_2)\bigr) \,. \]
  In particular, $\#I(\pi_1, \pi_2) = 2 \#I(\pi'_1, \pi'_2) - 2$.

  Conversely, given $\pi'_1 \colon E'_1 \to \PP^1$ and $\pi'_2 \colon E'_2 \to \PP^1$
  such that $\#\bigl(\pi'_1(E'_1[2]) \cap \pi'_2(E'_2[2])\bigr) \ge 2$,
  there are double covers $\pi_1 \colon E_1 \to \PP^1$ and~$\pi_2 \colon E_2 \to \PP^1$
  and $\id \neq \alpha \in G(\pi_1) \cap G(\pi_2)$ such that $(\pi'_1, \pi'_2)$
  is obtained from~$(\pi_1, \pi_2)$ and~$\alpha$ in the way described above.
\end{proposition}

\begin{proof}
  First consider one double cover $\pi \colon E \to \PP^1$. Up to post-composing
  with an automorphism of~$\PP^1$, we can assume
  that one element of $G(\pi)$ is $x \mapsto -x$; then $E$ has the form
  \[  E \colon y^2 = (x^2 - s) (x^2 - t) \]
  (up to scaling~$y$) with $s \neq t$ and $s$, $t$ nonzero. Let $T \in E[2]$ be the point such
  that translation by~$T$ on~$E$ is given by $(x,y) \mapsto (-x,-y)$.
  Then the isogeny $\phi \colon E \to E' = E/\langle T \rangle$ is
  \[ (x, y) \mapsto (x^2, xy)\,, \qquad \text{with} \qquad E' \colon y^2 = x (x - s) (x - t) \,. \]
  Since $\phi$ is an isogeny, it follows that
  \[ \pi(E_{\tors}) = \{\xi \in \PP^1 : \xi^2 \in \pi'(E'_{\tors})\} \,, \]
  where $\pi' \colon E' \to \PP^1$ is the $x$-coordinate map. Also note
  that $\{0, \infty\} \subset \pi'(E'[2])$.

  Given $\pi_1$, $\pi_2$ and~$\alpha$ as in the statement, we can again assume
  that $\alpha(x) = -x$, so that the curves $E_1$ and~$E_2$
  have the form $E_j \colon y^2 = (x^2 - s_j) (x^2 - t_j)$.
  Let $E'_1$ and~$E'_2$ be as in the statement; by the above, we can take
  $E'_j \colon y^2 = x (x - s_j) (x - t_j)$, and we let $\pi'_j$ denote
  the $x$-coordinate map. The first claim follows (with $\beta \colon x \mapsto x^2$).

  Since both ramification points $0$ and~$\infty$ of the squaring map are in~$I(\pi'_1, \pi'_2)$,
  the claim on the sizes of the sets also follows.

  For the converse statement, we can assume that $E'_j \colon y^2 = x (x - s_j) (x - t_j)$
  by moving two of the common ramification points to $0$ and~$\infty$. Then it is
  clear that the construction can be reversed.
\end{proof}

For $(\pi_1, \pi_2)$ and $(\pi'_1, \pi'_2)$ as in Proposition~\ref{prop:main},
we have that
\[ P(\pi_1, \pi_2) = (2a, 2b) \qquad\text{and}\qquad P(\pi'_1, \pi'_2) = (a+2, b) \]
for some $(a, b) \in \{(0, 1), (0, 2), (1, 1)\}$.

We write $T(a, b)$ for the maximum of $\#I(\pi_1, \pi_2)$ over all
pairs $(\pi_1, \pi_2)$ of double covers such that $P(\pi_1, \pi_2) = (a, b)$.
By the uniform boundedness results mentioned in the introduction, this makes sense.
We then have the following relations.

\begin{corollary} \label{cor:explicit}
  \[ T(0, 4) = 2 T(2, 2) - 2 = 4 T(3, 1) - 6 \qquad\text{and}\qquad T(0, 2) = 2 T(2, 1) - 2 \,. \]
\end{corollary}

This suggests that the maximal size of $\#I(\pi_1, \pi_2)$ is
obtained when $P(\pi_1, \pi_2) = (0, 4)$.


\section{Relation with genus $2$ torsion packets} \label{S:relg2}

Let $C$ be a curve of genus $g \ge 2$. Recall that a \emph{torsion packet}
on~$C$ is a maximal subset $T \subset C$ such that the difference of any two
points in~$T$, considered as a point on the Jacobian variety~$J$ of~$C$,
has finite order. Raynaud~\cite{Raynaud} proved that a torsion packet is always finite
(this was the statement of the Manin-Mumford conjecture); a nice and
short proof, based on a deep result of Serre, can be found in~\cite{BakerRibet2003}.
When $C$ is a hyperelliptic curve, then its
\emph{hyperelliptic torsion packet} is the torsion packet
that contains the ramification points of the hyperelliptic double
cover $C \to\PP^1$. We now assume that $C$ has genus~$2$ (and is therefore
in particular hyperelliptic). The curve $C$ is \emph{bielliptic} if there is a double
cover $\psi \colon C \to E$ with $E$ an elliptic curve. This is equivalent to the
existence of an ``extra involution''~$\alpha$ in~$\Aut(C)$, i.e.,
an involution distinct from the hyperelliptic involution~$\iota$, which is
the involution associated to the hyperelliptic double cover $C \to \PP^1$.
Then $E = C/\langle \alpha \rangle$, and there is another double
cover $\psi' \colon C \to C/\langle \alpha \iota \rangle = E'$, where $E'$ is also
an elliptic curve. The involution~$\alpha$ induces an involution
of~$\PP^1$ (since $\iota$ is central in the automorphism group of~$C$),
which we can take to be $x \mapsto -x$. In this case, $C$ can be given by an
equation of the form
\[ C \colon y^2 = (x^2 - u) (x^2 - v) (x^2 - w) \]
with $u$, $v$, $w$ distinct and nonzero, with $\alpha(x,y) = (-x,y)$.
Then we have
\begin{align*}
  E &\colon y^2 = (x - u) (x - v) (x - w) & & \text{with} \quad \psi(x, y) = (x^2, y) \,, \\
  E' &\colon y^2 = x (x - u) (x - v) (x - w) & & \text{with} \quad \psi'(x, y) = (x^2, x y)\,.
\end{align*}
There is an obvious birational morphism induced by $\psi$ and~$\psi'$,
\[ C \to \{(x, y, y') : (x, y) \in E,  (x, y') \in E'\} = E \times_{\PP^1} E' \,, \]
where the morphisms $\pi \colon E \to \PP^1$, $\pi' \colon E' \to \PP^1$
in the fibered product are the $x$-coordinate maps. We see that $P(\pi, \pi') = (3, 1)$.
(The morphism $C \to E \times_{\PP^1} E'$ is injective outside the ramification
points of the hyperelliptic double cover, which are identified in pairs.)

\begin{proposition} \label{prop:g2}
  In the situation described above, let $\rho \colon C \to \PP^1$ be the
  natural map obtained via $C \to E \times_{\PP^1} E' \to \PP^1$.
  Then the hyperelliptic torsion packet of~$C$ is the full preimage under~$\rho$
  of~$I(\pi, \pi')$. In particular, its size is
  \[ 4\#I(\pi, \pi') - 6 - 2 \#\bigl(I(\pi, \pi') \cap \{0, \infty\}\bigr) \,. \]
\end{proposition}

\begin{proof}
  Under the map $\psi \times \psi' \colon C \to E \times E'$,
  the fixed points of~$\iota$ on~$C$ are mapped to $2$-torsion points
  on $E \times E'$.
  The map $C \to E \times E'$ induces an isogeny $J \to E \times E'$
  such that the following diagram commutes, where the map on the lower left
  is the diagonal inclusion.
  \[ \xymatrix{C \ar[rr] \ar[drr]^{\psi \times \psi'} \ar[d]_{\rho} & & J \ar[d] \\
                \PP^1 \ar[r]_-{\operatorname{diag}} & \PP^1 \times \PP^1 & E \times E' \ar[l]^-{(\pi, \pi')}}
  \]
  Here the embedding of $C$ into~$J$ is $P \mapsto [P - W]$, where $W$ is a fixed
  ramification point, and we take the origin on $E \times E'$ to be the image of~$W$.

  If $P$ is a point in the hyperelliptic torsion packet,
  then $[P - W]$ is torsion in~$J$. This implies that $(\psi(P), \psi'(P))$
  is torsion on $E \times E'$, so that $\rho(P) \in I(\pi, \pi')$.

  Conversely, consider $\xi \in I(\pi, \pi')$. Then there are torsion points
  $P \in E$ and $P' \in E'$ such that $\pi(P) = \pi'(P') = \xi$.
  Let $Q \in C$ be a point with $\rho(Q) = \xi$. Then $(\psi(Q), \psi'(Q)) = (\pm P, \pm P')$,
  and so $(\psi(Q), \psi'(Q))$ is torsion. But then $[Q - W] \in J$
  must be torsion as well, since its image under the isogeny $J \to E \times E'$
  is torsion. So $Q$ is in the hyperelliptic torsion packet of~$C$.

  For the last statement, note that $\rho$ (which is $(x,y) \mapsto x^2$)
  has degree~$4$ and ramifies
  \begin{enumerate}[(1)]
    \item at the four points on~$C$ with $x = 0$ or $x = \infty$, and
    \item at the six ramification points of the hyperelliptic double cover,
          which map two-to-one onto $\pi(E[2]) \cap \pi'(E'[2])$.
    \qedhere
  \end{enumerate}
\end{proof}

We can reverse this construction.

\begin{proposition} \label{prop:conv}
  Let $\pi \colon E \to \PP^1$ and $\pi' \colon E' \to \PP^1$ be
  double covers such that $P(\pi, \pi') = (3, 1)$. Then there is a
  bielliptic genus~$2$ curve~$C$ such that $C$, $E$ and~$E'$ fit into
  a diagram as in the proof of Proposition~\ref{prop:g2}.
\end{proposition}

\begin{proof}
  By moving the fourth branch point of~$\pi$ to~$\infty$ and the
  fourth branch point of~$\pi'$ to~$0$, we can assume that
  $E$ and~$E'$ are as above, with $\pi$ and~$\pi'$ the $x$-coordinate
  maps, i.e.,
  \[ E \colon y^2 = (x - u) (x - v) (x - w) \qquad\text{and}\qquad
     E' \colon y^2 = x (x - u) (x - v) (x - w) \,.
  \]
  Then
  \[ C \colon y^2 = (x^2 - u) (x^2 - v) (x^2 - w) \]
  is the required curve.
\end{proof}

We denote by~$T$ the maximal size of a hyperelliptic torsion packet
on a bielliptic curve of genus~$2$.

As mentioned in the introduction, the correspondence established
in Propositions \ref{prop:g2} and~\ref{prop:conv} is used by
DeMarco, Krieger, and Ye~\cite{DeMarcoKriegerYe}*{\S9} to deduce a bound on~$T$
from a bound~$B$ on~$\#I(\pi_1, \pi_2)$ when $P(\pi_1, \pi_2) = (3, 1)$.
Their bound is $T \le 16 B$; we can improve this as follows.

\begin{corollary} \label{cor:I-and-T}
  \[ 4 T(3,1) - 6 \ge T \qquad\text{and therefore}\qquad T(0,4) \ge T \ge 34 \,. \]
\end{corollary}

\begin{proof}
  The first statement follows from the previous two propositions.
  The first inequality in the second statement then follows by Corollary~\ref{cor:explicit},
  and the last inequality comes from the example of
  Theorem~\ref{thm:g2_34}.
\end{proof}

It is clear that Corollary \ref{cor:I-and-T} implies Theorem \ref{Trel} immediately.

\begin{proof}[Proof of Theorem \ref{T22}]
  By the main result of~\cite{Poonen}, there are infinitely many
  (pairwise non-isomorphic) bielliptic genus~$2$ curves with a hyperelliptic
  torsion packet of size at least~$22$. By Proposition~\ref{prop:g2} and
  the construction in the proof of Proposition~\ref{prop:main}
  (applied twice backwards), we obtain corresponding pairs $(\pi_1, \pi_2)$.
\end{proof}

Theorem \ref{thm:g2_34} gives an example in Poonen's family with a hyperelliptic torsion packet of size $34$. By the same argument, we can show that there exists $(\pi_1, \pi_2)$ such that $\#I(\pi_1, \pi_2)=34$. This gives a proof of Theorem \ref{T34}, except for the claim on the orders of the torsion points with common $x$-coordinate. We will provide this missing part with more computational details at the end of Section \ref{S:subfamily}.


\section{Torsion packets on curves of genus $2$} \label{S:tp}

In this section, we explain how we found the example~\cite{Stoll34}
of a hyperelliptic torsion packet of size~$34$ on a curve of genus~$2$.
The approach is based on Poonen's result in~\cite{Poonen}, where he shows
that there are infinitely many essentially distinct curves of genus~$2$
with hyperelliptic torsion packets of size at least~$22$. We first give
a rough sketch of the idea behind Poonen's result.

One important fact is that the hyperelliptic torsion packet of a
hyperelliptic curve~$C$ is invariant under the automorphism group of~$C$.
This is because the automorphism group fixes the set of ramification points
of the hyperelliptic double cover and the difference of any two ramification
points is a torsion point (of order dividing~$2$). This implies that the
points in the hyperelliptic torsion packet come in orbits under the
automorphism group.

Now consider the moduli space of curves of genus~$2$. It has dimension~$3$.
We can (in principle) write down the condition that some torsion point
of given order~$n$ is in the image of the curve (we fix one of the ramification
points as the base point of the embedding of the curve into its Jacobian);
this gives a codimension-$1$ condition: we want to force one of finitely
many points on the Jacobian of dimension~$2$ to lie on a subvariety of
dimension $1$ and hence also codimension~$1$. So we can expect to find curves
with a hyperelliptic torsion packet of size at least $6 + 3 \cdot 2 = 12$,
where the first summand counts the ramification points, which are always
in the hyperelliptic torsion packet, and the second comes from the idea
that we can impose three independent torsion points onto the curve,
each giving us another one for free, since the hyperelliptic involution
is a nontrivial automorphism. Barring accidents (or ``unlikely intersections''),
we would not expect more than that. Such accidents do indeed occur,
as the following example shows.

\begin{example}
  The curve
  \[ C \colon y^2 = 4 x^6 - 12 x^5 - 3 x^4 + 46 x^3 - 15 x^2 - 24 x + 40 \]
  has minimal (geometric) automorphism group and a hyperelliptic
  torsion packet of size~$18$ (which is larger by~$6$ than the number
  we expect to find infinitely often).

  The first statement can be checked by looking at the invariants
  of~$C$; for the second we can use Poonen's program for computing
  torsion packets described in~\cite{Poonen2001}.
\end{example}

We can try to do better than what we can expect in the generic case
by considering subfamilies of
genus~$2$ curves that have a larger automorphism group. This gives us
more points ``for free'' for any torsion point we get on the curve.
On the other hand, the corresponding moduli spaces have smaller dimension,
so we cannot force as many orbits of torsion points on the curve as
in the generic case. These considerations lead to the following table,
which lists the relevant data for each possible automorphism group.
We describe it by specifying the reduced automorphism group
$\Aut(C)/\langle \iota \rangle$, which is a group of automorphisms of~$\PP^1$.

\[ \renewcommand{\arraystretch}{1.25}
   \begin{array}{|c|c|c|c|c|c|}\hline
     f               & \Aut(C)/\langle \iota \rangle & \#\!\Aut(C) & \dim M & \#T_{\min} & \#T \\\hline
     \text{generic}  & \{\id\} &  2 &      3 &          6 & 12 + 2 \delta \\
     x^6 + s x^4 + t x^2 + 1 & C_2  & 4 &  2 &          6 & 14 + 4 \delta \\
     x^5 + t x^3 + x & C_2 \times C_2 & 8 & 1 &         6 & 14 + 8 \delta \\
     x^6 + t x^3 + 1 & S_3     & 12 &      1 &         10 & \mathbf{22} + 12 \delta \\\hline
     x^6 + 1         & D_6     & 24 &      0 &         10 & 10 \\
     x^5 + x         & S_4     & 48 &      0 &         22 & \mathbf{22} \\
     x^5 + 1         & C_5     & 10 &      0 &         18 & 18 \\\hline
   \end{array}
\]

The curve can be given by an equation of the form $y^2 = f(x)$. $\dim M$
is the dimension of the moduli space of curves admitting (at least) the
given automorphism group, $\#T_{\min}$ is the size of the generic hyperelliptic
torsion packet in the family, and
\[ \#T = \#T_{\min} + (\dim M + \delta) \cdot \#\Aut(C) \]
is the size we can expect after forcing $\dim M$ orbits of torsion points;
$\delta$ counts the number of
additional full orbits we may be able to obtain ($\delta$ can be non-integral
when there are nontrivial stabilizers).

The second line, with reduced automorphism group~$C_2$, corresponds to
the bielliptic curves that feature in Section~\ref{S:relg2}. We can
have $\delta > 0$ in this case, too.

\begin{example}
  The curve
  \[ C \colon y^2 = (-9 \sqrt{3} + 16) x^6 + (-63 \sqrt{3} + 113) x^4 + (13 \sqrt{3} - 38) x^2 + 3 \sqrt{3} + 9 \]
  has $\Aut(C) \cong C_2 \times C_2$ (and so is generic bielliptic)
  and a hyperelliptic torsion packet of size at least~$18$.

  The claim on~$\Aut(C)$ can be shown in a similar way as for the preceding
  example. The torsion packet contains the six ramification points
  and the points with $x$-coordinates $\pm\sqrt{3}$, $\pm\sqrt{3}/3$, and $\pm(\sqrt{3}+2)$.
  We cannot use Poonen's program for this curve, since it requires the
  curve to be defined over~$\Q$.
\end{example}

The last three lines in the table correspond to a single point each
in the moduli space of curves of genus~$2$. The most interesting case
is the family with reduced automorphism group~$S_3$, which has one parameter
and has the additional benefit that there are four additional torsion points
on the curve throughout the family: the points at infinity and the points
with $x$-coordinate zero give points of order dividing~$6$ (they form an
orbit of size~$4$).
So we get $10$ points as our baseline, and we should be able to force one
full orbit of size~$12$ of torsion points on the curve in addition.
This is precisely what Poonen proves.

Note that the curves in this family are also bielliptic (an extra
involution is given by $(x, y) \mapsto (1/x, y/x^3)$). So we can use
the correspondence between torsion packets and sets $I(\pi_1, \pi_2)$
that is described in Section~\ref{S:relg2}. What we do is essentially
the following. For each $n \ge 3$ up to some bound, we compute the $n$-division
polynomials $h_{1,n}(t,x)$ and~$h_{2,n}(t,x)$ of the two elliptic curves
(which depend on the parameter~$t$). Then we compute the resultants
$R_{m,n}(t) = \Res_x\bigl(h_{1,m}(t,x), h_{2,n}(t,x)\bigr)$, which are
rational functions (in fact, polynomials) in~$t$. A root $t \neq \pm 2$
of~$R_{m,n}$ then gives us a parameter value such that the corresponding
curve~$C_t$ has an additional orbit of points in its hyperelliptic
torsion packet. If the orbit does not contain fixed points of some
nontrivial automorphism, then the torsion packet has size at least
$10 + 12 = 22$.

We now search for common irreducible factors among the various~$R_{m,n}$.
If we find such a common factor, then its roots give us curves~$C_t$
with \emph{two} additional orbits in the hyperelliptic torsion packet.
And indeed, we do find one such coincidence, which gives the curve
given in~\cite{Stoll34}. We state this result.

\begin{theorem} \label{thm:g2_34}
  The curve
  \[ C \colon y^2 = x^6 + 130 x^3 + 13 \]
  has a hyperelliptic torsion packet of size~$34$. It consists of the
  ramification points for the hyperelliptic double cover, the points
  at infinity, the points with $x$-coordinate zero, and the points
  whose $x$-coordinates satisfy the equation
  \[ x^{12} - 91 x^9 - 273 x^6 - 1183 x^3 + 169 = 0 \,. \]
\end{theorem}

\begin{proof}
  This can be shown using Poonen's pari/gp program for computing
  hyperelliptic torsion packets on genus~$2$ curves defined over~$\Q$;
  see~\cite{Poonen2001}.
\end{proof}


\section{Poonen's family} \label{S:subfamily}

As mentioned above, in \cite{Poonen}, Poonen considers the subfamily
\[ C_t \colon y^2 = (x^3 - 1)(x^3 - t^{12}) \]
of the family of all bielliptic curves of genus~$2$
and shows that there are infinitely many~$t$ such that the hyperelliptic torsion
packet of~$C_t$ has at least~$22$ points. (Here we use the $12$-th power to
avoid radicals in the sequel.) More precisely, let
\[ \iota \colon (x,y) \mapsto (x,-y)\,, \quad
   \sigma \colon (x,y) \mapsto (\zeta_{3}x,y)\,, \quad
   \tau \colon (x,y) \mapsto \left(\frac{t^{4}}{x},\frac{t^{6}y}{x^{3}}\right) \,,
\]
where $\zeta_3$ is a primitive third root of unity;
these are automorphisms of~$C_t$ satisfying
\[ \iota^2 = \sigma^3 = \tau^2 = \id\,, \quad
   \iota \tau = \tau \iota\,, \quad
   \iota \sigma = \sigma \iota \quad\text{and}\quad \sigma \tau = \tau \sigma^2 \,.
\]
As explained in the previous section, the $22$~points are
\begin{enumerate}[(1)]
  \item the six Weierstrass points (the fixed points of~$\iota$),
  \item the four points $0^+ = (0, t^6)$, $0^- = (0, -t^6)$, $\infty^+$, $\infty^-$
        (the fixed points of~$\sigma$), and
  \item a full length orbit of the group $\langle \iota, \sigma, \tau \rangle$ of order~$12$.
\end{enumerate}

As explained earlier in this paper, from a bielliptic curve~$C$ of genus~$2$,
we can obtain a pair~$(\pi_1, \pi_2)$ with $P(\pi_1, \pi_2) = (0, 4)$
and such that the size of~$I(\pi_1, \pi_2)$ is usually the same as the
size of the hyperelliptic torsion packet of~$C$; see Corollary~\ref{cor:I-and-T}.
The goal of this section is to explain which pairs correspond to the
curves~$C_t$ in Poonen's family, and how the effect of the large
automorphism group of~$C_t$ is reflected in the structure of~$I(\pi_1, \pi_2)$.

Since $C_t$ is bielliptic, by the construction in Section \ref{S:relg2} we have two morphisms
\begin{align*}
  \phi_{1} \colon C_{t} & \To C_{t}/\left\langle \tau \right\rangle
                           \simeq E_{1} \colon y^{2}=(x-u)(x-v)(x-w) \\
   (x,y) & \longmapsto \left(\frac{(x-t^{2})^{2}}{(x+t^{2})^{2}},
                             \frac{8t^{3}y}{(t^{6}+1)(x+t^{2})^{3}}\right)
\end{align*}
and
\begin{align*}
  \phi_{2} \colon C_{t} & \To C_{t}/\left\langle \tau\iota \right\rangle
                           \simeq E_{2} \colon y^{2}=x(x-u)(x-v)(x-w)\\
  (x,y) & \longmapsto\left(\frac{(x-t^{2})^{2}}{(x+t^{2})^{2}},
                           \frac{8t^{3}(x-t^{2})y}{(t^{6}+1)(x+t^{2})^{4}}\right) \,.
\end{align*}
Here we choose the origin of~$E_1$ to be~$\infty$, and the origin of~$E_2$ to be~$(0,0)$.

We abuse notation and also write~$\sigma$ for the automorphism of~$J(C_t)$
induced by~$\sigma$. We use~$\theta_j$ to denote the homomorphism $J(C_t) \to E_j$
induced by~$\phi_j$, and we write $\theta_j^\vee \colon E_j \to J(C_t)$
for the map induced by pull-back of divisors under~$\phi_{j}$.

\begin{proposition} \label{isogeny}
  Let $T_1 = \phi_{1}(0^+)$ and $T_2 = \phi_{2}(0^+)$. Then $T_1$ and~$T_2$
  have order~$3$. Moreover,
  \[ \varphi_{12} = \theta_{2} \circ \sigma \circ \theta_{1}^\vee \colon E_1 \To E_2
     \qquad\text{and}\qquad
     \varphi_{21} = \theta_{1} \circ \sigma \circ \theta_{2}^\vee \colon E_2 \To E_1
  \]
  are isogenies with $\ker(\varphi_{12}) = \langle T_1 \rangle$ and
  $\ker(\varphi_{21}) = \langle T_2 \rangle$.
\end{proposition}

\begin{proof}
  It is clear that $\varphi_{12}$ is a homomorphism of elliptic curves.
  Let $P \in E_1$ and choose $Q, R \in C_t$ such that
  $\phi_{1}(Q) = P$ and $\phi_{1}(R) = \infty$
  (so $R = (-t^2, \pm t^3 (t^6+1))$ and $\iota \tau R = R$).
  Then, denoting the linear equivalence class of a divisor~$D$ on~$C_t$ by~$[D]$,
  \begin{align*}
    \varphi_{12}(P)
      &= (\theta_2 \circ \sigma)([Q + \tau Q - R - \tau R])
       = \theta_2([\sigma Q + \sigma \tau Q - \sigma R - \sigma \tau R]) \\
      &= \theta_2([\sigma Q + \tau \sigma^2 Q - \sigma R - \sigma \tau R])
       \stackrel{(*)}{=} \theta_2([\sigma Q - \iota \tau \sigma^2 Q - \sigma R + \iota \sigma \tau R]) \\
      &= \theta_2([\sigma Q - \iota \tau \sigma^2 Q - \sigma R + \sigma \iota \tau R])
       = \phi_{2}(\sigma Q) - \phi_{2}(\sigma^2 Q) \,.
  \end{align*}
  where at~$(*)$ we use that $[Q_1 - Q_2] = [-\iota Q_1 + \iota Q_2]$ in~$J(C_t)$,
  and for the last equality that $\iota \tau R = R$.
  So $P \in \ker(\varphi_{12})$ if and only if $\sigma Q = \sigma^2 Q$ or
  $\sigma Q = \tau \iota \sigma^2 Q = \sigma \tau \iota Q$;
  equivalently, $Q = \sigma Q$ or $Q = \tau \iota Q$. In the first case,
  $P = \pm T_1$, and in the second case, $P = \infty$. This also implies
  that $T_1$ has order~$3$.
  The claim regarding $\varphi_{21}$ follows in the same way.
\end{proof}

We will also need the following.

\begin{lemma} \label{L:help}
  With the notations introduced above, $\theta_j \circ \sigma \circ \theta_j^\vee$
  is multiplication by~$-1$ on~$E_j$.
\end{lemma}

\begin{proof}
  We first observe that all divisors of the form $Q + \sigma Q + \sigma^2 Q$ on~$C_t$
  are linearly equivalent. This is because the quotient $C_t/\langle \sigma \rangle$
  is the curve $y^2 = (x - 1)(x - t^{12})$ of genus zero.

  Let $P \in E_1$ and choose $Q, R \in C_t$ as in the preceding proof. Then
  \begin{align*}
    (\theta_1 \circ \sigma \circ \theta_1^\vee)(P)
      &= (\theta_1 \circ \sigma)([Q + \tau Q - R - \tau R])
       = \theta_1([\sigma Q + \sigma \tau Q - \sigma R - \sigma \tau R]) \\
      &= \theta_1([\sigma Q + \tau \sigma^2 Q - \sigma R - \tau \sigma^2 R])
       = \theta_1([\sigma Q + \sigma^2 Q - \sigma R - \sigma^2 R]) \\
      &= \theta_1([R - Q]) = -P \,.
  \end{align*}
  The argument for~$E_2$ is identical.
\end{proof}

Now let
\[ E_{s} \colon y^{2} = (x^{2}-s^{2})\left(x^{2}-\frac{1}{s^{2}}\right) \quad\text{with origin $(s,0)$,} \]
and
\[ E_{s}' \colon y^{2} = x(x-1)\left(x-\frac{(s^{2}+1)^{2}}{4s^{2}}\right)
                    \quad\text{with origin $\displaystyle\left(\frac{(s^{2}+1)^{2}}{4s^{2}},0\right)$.} \]
We have an isogeny
\[ \psi \colon E_{s} \To E_{s}'\,, \qquad
               (x,y) \longmapsto \left(\frac{(x^{2}+1)^{2}}{4x^{2}}, \frac{(x^{4}-1)y}{8x^{3}}\right)\,.
\]
Note that $\ker(\psi) = E_{s}[2]$ (in particular, $E_s$ and $E'_s$ are isomorphic
as elliptic curves) and $\psi$ is independent of~$s$.

For any~$t$, we have the two covering maps $\phi_{j} \colon C_{t} \to E_{j}$.
By moving the three common $2$-torsion $x$-coordinates of $E_1$ and~$E_2$
to $\{0,1,\infty\}$, we can assume that $E_{j} = E_{s_{j}}'$
for suitable $s_1$ and~$s_2$. One of the four preimages under~$\psi$ of the point~$T_j$
of order~$3$ will also have order~$3$; let $S_j \in E_{s_j}$ be this point.
Since $T_1$ and~$T_2$ have the same $x$-coordinate (this is still true after
applying an automorphism of~$\PP^1$), the $x$-coordinates of $S_1$ and~$S_2$
are equal up to sign and/or taking inverses; by replacing $s_2$ by $\pm s_2$
or $\pm 1/s_2$, we can make sure that $x(S_1) = x(S_2)$.

Conversely, if $S_{1} \in E_{s_{1}}[3]$ and $S_{2} \in E_{s_{2}}[3]$ have the
same $x$-coordinate, then $C = E_{s_{1}}' \times_{\mathbb{P}^{1}} E_{s_{2}}'$
contains four points $(\pm\psi(S_{1}), \pm\psi(S_{2}))$ whose differences have
order~$3$ in~$J(C)$ and form two orbits under the hyperelliptic involution.
Any such curve has a model of the form~$C_t$.

For $j = 1,2$, $\psi$ can be decomposed as $\psi = [2] \circ \lambda_{s_j}$,
where $\lambda_{s_j}$ is an isomorphism from $E_{s_j}$ to~$E_{s_j}'$ such that
$\lambda_{s_j}(S_j) = -T_j$. Proposition~\ref{isogeny} shows that
there are isogenies $\varphi_{12}$ and~$\varphi_{21}$ between $E_{s_1}'$ and~$E_{s_2}'$.
We use the same symbols for the induced isogenies between $E_{s_1}$ and~$E_{s_2}$.
\[ \xymatrix{ E_{s_{1}} \ar[r]^{\lambda_{s_{1}}} \ar[d]_{\varphi_{12}}
                & E_{s_{1}}' \ar[r]^{[2]} \ar[d]^{\varphi_{12}}
                & E_{s_{1}}' \ar[d]
                & & E_{s_{2}} \ar[r]^{\lambda_{s_{2}}} \ar[d]_{\varphi_{21}}
                & E_{s_{2}}' \ar[r]^{[2]} \ar[d]^{\varphi_{21}}
                & E_{s_{2}}'\ar[d] \\
              E_{s_{2}} \ar[r]^{\lambda_{s_{2}}}
                & E_{s_{2}}' \ar[r]^{[2]}
                & E_{s_{2}}'
                & & E_{s_{1}} \ar[r]^{\lambda_{s_{1}}}
                & E_{s_{1}}' \ar[r]^{[2]}
                & E_{s_{1}}'
            }
\]

\begin{proposition}
  Let $\pi_1$ and~$\pi_2$ be the $x$-coordinate maps of $E_{s_1}$ and~$E_{s_2}$.
  Assume that there
  are $S_1 \in E_{s_1}[3]$ and $S_2 \in E_{s_2}[3]$ with $\pi_1(S_1) = \pi_2(S_2)$.
  Then for any $P_{1} \in E_{s_{1}}$ and $P_{2} \in E_{s_{2}}$ with
  $\pi_1(P_{1}) = \pi_2(P_{2})$, we have
  \[ \pi_1\bigl([2]^{-1}(-P_{1}\pm\varphi_{21}(P_{2}))\bigr)
       = \pi_2\bigl([2]^{-1}(-P_{2}\pm\varphi_{12}(P_{1}))\bigr)\,.
  \]
\end{proposition}

This gives usually eight further common torsion $x$-coordinates, in addition
to the set $\pi_1(P_1 + E_1[2]) = \pi_2(P_2 + E_2[2])$ of size four, corresponding
to an orbit on~$C_t$ under~$\langle \iota, \sigma, \tau \rangle$.

\begin{proof}
  We shall use the same notations as in the proof of Proposition~\ref{isogeny},
  except that now we replace $E_1$ and~$E_2$ with $E_{s_1}'$ and~$E_{s_2}'$.

  Let $Q_1 = \lambda_{s_1}(P_1)$ and $Q_2 = \lambda_{s_2}(P_2)$.
  Then $[2]Q_1$ and $[2]Q_2$ have the same $x$-coordinate.
  So there is $Q \in C_t$ such that $(\phi_{1}(Q),\phi_{2}(Q)) = ([2]Q_1,[2]Q_2)$.
  Let $(R_1,R_2) \in E_{s_1}' \times E_{s_2}'$ be such that
  \[ (\theta_{1},\theta_{2})^{\vee}(R_1,R_2) = [Q-W] \,, \]
  where $W \in C_t$ is a fixed Weierstrass point.
  Then
  \begin{align*}
    ([2]R_1,[2]R_2) &= \bigl((\theta_{1},\theta_{2}) \circ (\theta_{1},\theta_{2})^{\vee}\bigr)(R_{1},R_{2}) \\
                    &= (\theta_{1},\theta_{2})([Q-W])
                     = ([2]Q_1,[2]Q_2) - (\phi_{1}(W),\phi_{2}(W))\,.
  \end{align*}
  Therefore,
  $(R_1,R_2) = (Q_1,Q_2) - (W_1,W_2)$
  for some $(W_1,W_2) \in (E_{s_1}' \times E_{s_2}')[4]$. Since
  by Proposition~\ref{isogeny} and Lemma~\ref{L:help}
  \begin{align*}
    (-Q_{1} + \varphi_{21}(Q_{2})) - (-W_{1} + \varphi_{21}(W_{2}))
      &= \bigl(\theta_{1} \circ \sigma \circ (\theta_{1},\theta_{2})^{\vee}\bigr)(R_{1},R_{2}) \\
      &= (\theta_{1} \circ \sigma) ([Q-W])
       = \phi_{1}(\sigma Q) - \phi_{1}(\sigma W)\,,
  \end{align*}
  we have
  \[ -Q_{1} + \varphi_{21}(Q_{2}) - \phi_{1}(\sigma Q)
      = -W_{1} + \varphi_{21}(W_{2}) - \phi_{1}(\sigma W) \in E_{s_1}'[4]\,.
  \]
  Therefore, $(P_1,P_2) \mapsto -Q_{1} + \varphi_{21}(Q_{2}) - \phi_{1}(\sigma Q)$
  gives a constant map from $E_{s_1} \times_{\PP^1} E_{s_2}$ to~$E_{s_1}'$.
  To determine this constant, we take $(P_1,P_2) = (S_1,S_2)$;
  then $(Q_1,Q_2) = (-T_1,-T_2)$ and $Q = 0^+$, which imply
  \[ \phi_{1}(\sigma Q) = \phi_1(Q) = T_1 = -Q_{1} + \varphi_{21}(Q_{2}) \,, \]
  showing that the constant is in fact zero.
  Similarly, we have
  \[ \phi_{1}(\sigma^2 Q) = -Q_{1} - \varphi_{21}(Q_{2})\,, \quad
     \phi_{2}(\sigma Q) = -Q_{2} + \varphi_{12}(Q_{1})\,, \quad
     \phi_{2}(\sigma^2 Q) = -Q_{2} - \varphi_{12}(Q_{1})\,.
  \]
  Now consider a point $P'_1 \in E_{s_1}$ such that $[2]P'_1 = -P_1 \pm \varphi_{21}(P_2)$.
  Then
  \[ \psi(P'_1) = \lambda_{s_1}([2]P'_1) = -Q_1 \pm \varphi_{21}(Q_2) = \phi_1(\sigma^k Q) \]
  (with $k = 1$ for the positive sign and $k = 2$ for the negative sign).
  Since $\phi_{1}(\sigma^k Q)$ and~$\phi_{2}(\sigma^k Q)$ have the same $x$-coordinate,
  there is $P'_2 \in E_{s_2}$ with $[2]P'_2 = -P_2 \pm \varphi_{12}(P_1)$ such that
  $\psi(P'_2)$ has the same $x$-coordinate as~$\psi(P'_1)$. Since $E_{s_j}[2]$
  acts transitively on the fibers of~$\psi$ (and $\psi$ is independent of~$s_j$),
  the conclusion follows.
\end{proof}

Finally, we give the proof of Theorem \ref{T34}.

\begin{proof}[Proof of Theorem \ref{T34}]
Let $C:y^{2}=x^{6}+130x^{3}+13$ be the curve given in Theorem \ref{thm:g2_34},
and let
\begin{align*}
\phi_{1}:C & \longrightarrow E_{1}':y^{2}=\left(x-\frac{17+4\sqrt{13}}{9}\right)\left(x^{2}-\frac{7-\sqrt{13}}{6}x+\frac{31-7\sqrt{13}}{18}\right)\\
(x,y) & \longmapsto\left(\frac{(x-\sqrt[6]{13})^{2}}{(x+\sqrt[6]{13})^{2}},\frac{cy}{(x+\sqrt[6]{13})^{3}}\right)
\end{align*}
and
\begin{align*}
\phi_{2}:C & \longrightarrow E_{2}':y^{2}=x\left(x-\frac{17+4\sqrt{13}}{9}\right)\left(x^{2}-\frac{7-\sqrt{13}}{6}x+\frac{31-7\sqrt{13}}{18}\right)\\
(x,y) & \longmapsto\left(\frac{(x-\sqrt[6]{13})^{2}}{(x+\sqrt[6]{13})^{2}},\frac{c(x-\sqrt[6]{13})y}{(x+\sqrt[6]{13})^{4}}\right)
\end{align*}

be the two morphisms constructed in Section \ref{S:relg2}, where $c$ is a constant
whose precise value is not important for us. We describe the images
of the hyperelliptic torsion packet of $C$ under $\phi_{1}$ and
$\phi_{2}$ as follows:
\begin{enumerate}
\item the six Weierstrass points are mapped to three torsion points of order
$2$;
\item the four points $0^{+},0^{-},\infty^{+},\infty^{-}$ are mapped to two
torsion points of order $3$;
\item the remaining points in the hyperelliptic torsion packet of $C$ are
mapped to
\begin{itemize}
\item the points with $x$-coordinates satisfying
\[
9x+(5-2\sqrt{13})=0,
\]
which are torsion points of order $4$ on $E_{1}'$ and of order $6$
on $E_{2}'$;
\item the points with $x$-coordinates satisfying
\[
3x-(5-2\sqrt{13})=0,
\]
which are torsion points of order $6$ on $E_{1}'$ and of order $4$
on $E_{2}'$;
\item the points with $x$-coordinates satisfying
\[
9x^{2}+6(5-2\sqrt{13})x+(17+4\sqrt{13})=0,
\]
which are torsion points of order $8$ on $E_{1}'$ and of order $24$
on $E_{2}'$;
\item the points with $x$-coordinates satisfying
\[
9x^{2}+18(7+2\sqrt{13})x+(17+4\sqrt{13})=0,
\]
which are torsion points of order $24$ on $E_{1}'$ and of order
$8$ on $E_{2}'$.
\end{itemize}
\end{enumerate}
In particular, all these torsion points have order dividing $24$.
By moving the three common $2$-torsion $x$-coordinates of $E_1'$
and $E_2'$ to $\{0,1,\infty\}$, we can see that $E_1'\cong E_s'$
and $E_2'\cong E_{s/3}'$ for some $s\in\C$ satisfying
$s^8+174s^4+81=0$. By lifting the torsion points on $E_s'$ and
$E_{s/3}'$ to the torsion points on $E_1=E_s$ and $E_2=E_{s/3}$
via $\psi$, we obtain the conclusion.
\end{proof}


\end{document}